\documentclass[12pt]{amsart}

\usepackage{amsfonts,amssymb,amscd,amsmath,enumerate, verbatim,calc, hhline}
\newtheorem{theorem}{Theorem}[section]
\newtheorem{cor}[theorem]{Corollary}
\newtheorem{lemma}[theorem]{Lemma}

\theoremstyle{definition}

\newtheorem{remark}[theorem]{Remark}
\numberwithin{equation}{section}

\begin{document}

\title[New characterizations for the essential norms of...]
{New characterizations for the essential norms of generalized weighted composition operators\\ between Zygmund type spaces}

\author[Hassanlou and Sanatpour]{Mostafa Hassanlou and Amir H. Sanatpour}
\address{Engineering Faculty of Khoy, Urmia University, Urmia, Iran.}
\email{m.hassanlou@urmia.ac.ir}
\address{Department of Mathematics, Kharazmi University, Tehran, Iran.}
\email{a$\_ $sanatpour@khu.ac.ir, a.sanatpour@gmail.com}




\begin{abstract}
We give different types of new characterizations for the boundedness and essential norms of generalized weighted composition operators between Zygmund type spaces. Consequently, we obtain new characterizations for the compactness of such operators.
\end{abstract}

\maketitle
\section{\sc\bf Introduction}
\indent Let $\mathbb{D}$ denote the open unit ball of the complex plane $\mathbb{C}$ and
$H(\mathbb{D})$ denote the space of all complex-valued analytic functions on $\mathbb{D}$. By a \textit{weight} $\nu$ we mean
a strictly positive bounded function $\nu : \mathbb{D} \rightarrow \mathbb{R}^+$. The \textit{weighted-type space} $H_{\nu}^{\infty}$ consists of all functions $f \in H(\mathbb{D})$ such that
$$ \|f\|_{\nu} = \sup_{z \in \mathbb{D}} \nu(z) |f(z)| < \infty.$$
For a weight $\nu$, the \textit{associated weight} $\widetilde{\nu}$ is defined by
$$\widetilde{\nu}(z) = (\sup \{ |f(z)|: f \in H_{\nu}^{\infty}, \|f\|_{\nu} \leq 1 \})^{-1}.$$
It is known that for the \textit{standard weights} ${\nu}_{\alpha}(z) = (1-|z|^2)^{\alpha}$, $0 < \alpha < \infty$, and for the \textit{logarithmic weight}
${\nu}_{\log} (z)= \left ( \log \frac{2}{1-|z|^2} \right)^{-1}$, the associated weights and weights are the same.

For each $0<\alpha<\infty$, the \textit{Bloch type space} $\mathcal{B}_{\alpha}$ consists of all functions $f \in H(\mathbb{D})$ for which
$$\|f\|_{s\mathcal{B}_{\alpha}} =\sup_{z\in \mathbb{D}} (1-|z|^2)^{\alpha} |f'(z)| < \infty. $$
The space $\mathcal{B}_{\alpha}$ is a Banach space equipped with the norm
$$\|f\|_{\mathcal{B}_{\alpha}} = |f(0)| + \|f\|_{s\mathcal{B}_{\alpha}},$$
for each $f\in \mathcal{B}_{\alpha}$.
The \textit{little Bloch type space} $\mathcal{B}_{\alpha, 0}$ is the closed subspace of $\mathcal{B}_{\alpha}$ consists of those functions
$f\in {\mathcal{B}_\alpha}$ satisfying
$$\lim_{|z|\to 1} ( 1 - |z|^2)^\alpha |f'(z)|=0.$$
The classic \textit{Zygmund space} $\mathcal{Z}$ consists of all functions $f \in H(\mathbb{D})$ which are continuous on the closed unit ball
$ \overline{\mathbb{D}}$ and
$$\sup \frac{|f(e^{i(\theta +h)}) + f(e^{i(\theta -h)}) - 2 f (e^{i \theta})|}{h} < \infty,$$
where the supremum is taken over all $\theta \in \mathbb{R}$ and $h> 0$.
By \cite[Theorem 5.3]{duren1}, an analytic function $f$ belongs to $\mathcal{Z}$ if and only if $\sup_{z\in \mathbb{D}} (1-|z|^2) |f''(z)| < \infty$.
Motivated by this, for each  $0<\alpha<\infty$, the \textit{Zygmund type space} $\mathcal{Z}_{\alpha}$ is defined to be the space of all functions $f \in H(\mathbb{D})$ for which
$$\|f\|_{s\mathcal{Z}_{\alpha}} = \sup_{z\in \mathbb{D}} (1-|z|^2)^{\alpha} |f''(z)| < \infty.$$
The space $\mathcal{Z}_{\alpha}$ is a Banach space equipped with the norm
$$\|f\|_{\mathcal{Z}_{\alpha}} = |f(0)| + |f'(0)| + \|f\|_{s\mathcal{Z}_{\alpha}},$$
for each $f\in \mathcal{Z}_{\alpha}$. The \textit{little Zygmund type space} $\mathcal{Z}_{\alpha, 0}$ is the closed subspace of $\mathcal{Z}_{\alpha}$ consists of those functions $f\in {\mathcal{Z}_\alpha}$ satisfying
$$\lim_{|z|\to 1} ( 1 - |z|^2)^\alpha |f''(z)|=0.$$

Recall that for the Banach spaces $X$ and $Y$, the space of all bounded operators $T: X\rightarrow Y$ is denoted by ${\mathcal{B}}(X, Y)$ and the \textit{operator norm} of $T\in {\mathcal{B}}(X, Y)$ is denoted by $\|T\|_{X\rightarrow Y}$. 
The closed subspace of ${\mathcal{B}}(X, Y)$ containing all compact operators $T: X\rightarrow Y$ is denoted by ${\mathcal{K}}(X, Y)$. 
The \textit{essential norm} of $T\in {\mathcal{B}}(X, Y)$, denoted by $\|T\|_{e, X\rightarrow Y}$, is defined as the distance from $T$ to ${\mathcal{K}}(X, Y)$, that is
$$\|T\|_{e, X\rightarrow Y}=\inf\{\|T - K\|_{X\rightarrow Y} : K\in {\mathcal{K}}(X, Y)\}.$$
Clearly, an operator $T\in {\mathcal{B}}(X, Y)$ is compact if and only if $\|T\|_{e, X\rightarrow Y} = 0$. 
Therefore, essential norm estimates of bounded operators result in necessary and/or sufficient conditions for the compactness of such operators. 
Essential norm estimates of different types of operators between various classes of Banach spaces have been studied by many authors. 
See, for example, \cite{hu4, hu3, lin1, lin3, sanatpour1, x2} and references therein.

Let $u$ and $\varphi$ be analytic functions on $\mathbb{D}$ such that $\varphi (\mathbb{D}) \subseteq \mathbb{D}$.
The \textit{weighted composition operator} $uC_\varphi$ is defined by $uC_\varphi f = u\cdot f\circ \varphi$ for all $f \in H(\mathbb{D})$.
When $u = 1$ we get the well-known \textit{composition operator} $C_\varphi$ given by $C_\varphi f = f\circ \varphi$ for all $f \in H(\mathbb{D})$.
Weighted composition operators appear in the study of dynamical systems and also it is known that isometries
on many analytic function spaces are of the canonical forms of weighted composition operators.
Operator theoretic properties of (weighted) composition operators have been studied by many authors between different classes of analytic function spaces.
See, for example, \cite{lin1, lin3, mon, SH} and the references therein.

For each non-negative integer $k$, the \textit{generalized weighted composition operator} $D_{\varphi, u}^k$ is defined by
$$D_{ \varphi, u}^k f(z) = u(z) f^{(k)}(\varphi(z)),$$
for each $f \in H(\mathbb{D})$ and $z \in \mathbb{D}$.
The class of generalized weighted composition operators include weighted composition operators $uC_{\varphi}=D_{\varphi, u}^0 $, \textit{composition operators followed by differentiation} $D C_{\varphi}=D_{\varphi, \varphi'}^1$ and \textit{composition operators proceeded by differentiation} $C_{\varphi} D=D_{\varphi, 1}^1$ \cite{liu}.
Also, weighted types of operators $DC_\varphi$ and $C_\varphi D$ are of the form $D_{\varphi, u}^k$, that is 
$u DC_\varphi = D_{\varphi, u \varphi'}^1$ and $u C_\varphi D = D_{\varphi, u}^1$ \cite{long}.

Boundedness and compactness of generalized weighted composition operators have been studied between Bloch type spaces and Zygmund type spaces in \cite{hu4, li3}, and between Bloch type spaces and weighted-type spaces in \cite{li2, x3}. Essential norms of generalized weighted composition operators in these cases have been studied in \cite{hu4, hu3, sanatpour1}.
In \cite{x1}, different characterizations for the boundedness and compactness of these operators between Bloch type spaces are given and also their essential norms are investigated in \cite{x2}.
In this paper we first study boundedness of generalized weighted composition operators between Zygmund type spaces and give new characterizations for the boundedness of these operators. Then, we find estimates for the essential norms of such operators in terms of the new characterizations. Consequently, we obtain different types of characterizations for the compactness of such operators.

The following lemma, which will be used in the next chapters, collects some useful estimates for the functions in Zygmund type spaces. See, for example, \cite{sanatpour1} and references therein.
\begin{lemma} \label{l31}
For every $f \in \mathcal{Z}_{\alpha}$ we have
\begin{itemize}
\item[$(i)$] $|f'(z)| \leq \frac{2}{1-\alpha} \|f\|_{\mathcal{Z}_{\alpha}} $ and $|f(z)| \leq \frac{2}{1-\alpha} \|f\|_{\mathcal{Z}_{\alpha}}$
for $0 < \alpha < 1$,
\item[$(ii)$] $|f'(z)| \leq 2 \|f\|_{\mathcal{Z}} \log \frac{2}{1-|z|}$ and $|f(z)| \leq \|f\|_{\mathcal{Z}}$ for $\alpha =1$,
\item[$(iii)$] $|f'(z)| \leq \frac{2}{\alpha -1} \frac{\|f\|_{\mathcal{Z}_{\alpha}}}{(1-|z|)^{\alpha -1}}$, for $1< \alpha <\infty$,
\item[$(iv)$] $|f(z)| \leq \frac{2}{(\alpha -1)(2-\alpha)} \|f\|_{\mathcal{Z}_{\alpha}}$, for $1 < \alpha <2$,
\item[$(v)$] $|f(z)| \leq 2 \|f\|_{\mathcal{Z}_2} \log \frac{2}{1-|z|}$, for $\alpha =2$,
\item[$(vi)$] $|f(z)| \leq \frac{2}{(\alpha -1)(\alpha -2)} \frac{\|f\|_{\mathcal{Z}_{\alpha}}}{(1-|z|)^{\alpha -2}}$, for $2< \alpha <\infty$.
\end{itemize}
\end{lemma}
It is known that for each $n \geq 2$ and $0 < \alpha < \infty$ we have
$$ |f^{(n)} (z)| \leq \frac{\|f\|_{\mathcal{B}_{\alpha}}}{(1-|z|^2)^{\alpha + n -1}},$$
for all $f \in \mathcal{B}_{\alpha}$ and $z\in \mathbb{D}$, see \cite{zhu1}.
Therefore, for each $n \geq 2$ and $0 < \alpha < \infty$ we have
\begin{equation} \label{eq11}
|f^{(n +1)} (z)| \leq \frac{\|f\|_{\mathcal{Z}_{\alpha}}}{(1-|z|^2)^{\alpha + n -1}},
\end{equation}
for all $f \in \mathcal{Z}_{\alpha}$ and $z\in \mathbb{D}$.
Note that, by the definition of Zygmund type spaces, it is clear that \eqref{eq11} also holds in the case of $n=1$.

In this paper, for real scalars $A$ and $B$, the notation $A \lesssim B$ means $A \leq c B$ for some
positive constant $c$. Also, the notation $A \approx B$ means $A \lesssim B$ and $B \lesssim A$.
\section{\sc\bf Boundedness}
For each $a\in \mathbb{D}$, the following test functions in $H(\mathbb{D})$ will be used in our proofs
$$ f_a (z) = \frac{(1-|a|^2)^2}{(1-\overline{a}z)^{\alpha}}, \quad g_a (z) = \frac{(1-|a|^2)^3}{(1-\overline{a}z)^{\alpha +1}}, \quad h_a (z) = \frac{(1-|a|^2)^4}{(1-\overline{a}z)^{\alpha +2}}.$$
Moreover, in order to simplify the notations, we define
\begin{align*}
A(u,\varphi, \alpha, \beta, n) =  & \sup_{z \in \mathbb{D}} \frac{(1-|z|^2)^{\beta}}{(1-|\varphi(z)|^2)^{\alpha + n-2}} |u''(z)|, \\
B(u,\varphi, \alpha, \beta, n) = & \sup_{z \in \mathbb{D}} \frac{ (1-|z|^2)^{\beta} }{(1-|\varphi(z)|^2)^{\alpha + n -1}} |2u'(z) \varphi'(z) + u(z) \varphi''(z) |, \\
C(u , \varphi, \alpha, \beta, n) = &  \sup_{z \in \mathbb{D}} \frac{ (1-|z|^2)^{\beta} }{(1-|\varphi(z)|^2)^{\alpha + n}} |u(z) \varphi'^2(z)|.
\end{align*}
In the next theorem, we give three different characterizations for the boundedness of $ D_{\varphi,u}^n : \mathcal{Z}_{\alpha} \rightarrow \mathcal{Z}_{\beta} $.
\begin{theorem} \label{g34}
Let $u\in H(\mathbb{D})$, $\varphi$ be an analytic selfmap of $\mathbb{D}$ and $(n, \alpha) \neq (1, 1)$. Then, the following  statements are equivalent for each
$0< \beta < \infty$.
\begin{itemize}
\item[$(i)$] $ D_{\varphi,u}^n : \mathcal{Z}_{\alpha} \rightarrow \mathcal{Z}_{\beta} $ is bounded.
\item[$(ii)$] $\sup_{j \geq 1} j^{\alpha -2} \| D_{\varphi,u}^n I^{j+1}\|_{\mathcal{Z}_{\beta}} < \infty$, where $I^{j}(z) = z^j$ for each $j\geq 1$ and $z\in \mathbb{D}$.
\item[$(iii)$] $u \in \mathcal{Z}_{\beta}$ and
$$\sup_{z \in \mathbb{D}} (1-|z|^2)^{\beta} |u(z) \varphi'^2 (z)| < \infty,$$
$$\sup_{z \in \mathbb{D}} (1-|z|^2)^{\beta} |2u'(z) \varphi'(z) + u(z) \varphi''(z) | < \infty,$$
$$ \sup_{a \in \mathbb{D}}\| D_{\varphi,u}^n f_a\|_{\mathcal{Z}_{\beta}} < \infty, \quad
\sup_{a \in \mathbb{D}} \| D_{\varphi,u}^n g_a\|_{\mathcal{Z}_{\beta}} < \infty,
\quad \sup_{a \in \mathbb{D}} \| D_{\varphi,u}^n h_a\|_{\mathcal{Z}_{\beta}} < \infty.$$
\item[$(iv)$]
$$ \max \{A(u,\varphi, \alpha, \beta,n), B(u,\varphi, \alpha, \beta,n), C(u,\varphi, \alpha, \beta,n) \} < \infty.  $$
Moreover, this is also equivalent to
$$ \max \{\|u\|_{\mathcal{Z}_{\beta}},B(u,\varphi, \alpha, \beta,1), C(u,\varphi, \alpha, \beta,1) \} < \infty,$$
in the special case of $n=1$ and $0 < \alpha < 1$.
\end{itemize}
\end{theorem}
\begin{proof}
$(i) \Rightarrow (ii)$ Suppose that $n \geq 1$ and $ D_{\varphi,u}^n : \mathcal{Z}_{\alpha} \rightarrow \mathcal{Z}_{\beta} $ is a bounded operator. Since $\{j^{\alpha -1} I^j\}$ is a bounded sequence in  $\mathcal{B}_{\alpha}$, see for example \cite{zhao1}, one can see that $\{j^{\alpha -2} I^{j+1}\}$ is a bounded sequence in
$\mathcal{Z}_{\alpha}$. Therefore, boundedness of $D_{\varphi,u}^n : \mathcal{Z}_{\alpha} \rightarrow \mathcal{Z}_{\beta}$ implies that
$\sup_{j \geq 1} j^{\alpha -2} \| D_{\varphi,u}^n I^{j+1}\|_{\mathcal{Z}_{\beta}} < \infty$.  \\
$(ii) \Rightarrow (iii)$ Let $M=\sup_{j \geq 1} j^{\alpha -2} \| D_{\varphi,u}^n I^{j+1}\|_{\mathcal{Z}_{\beta}}<\infty$. 
Recalling the definition of $f_a$, $g_a$ and $h_a$, one can see that
$$f_a (z) = (1-|a|^2)^2 \sum_{j=0}^{\infty} \frac{\Gamma (j+ \alpha)}{j! \Gamma (\alpha)} \overline{a}^j z^j, \quad
 g_a (z) = (1-|a|^2)^3 \sum_{j=0}^{\infty} \frac{\Gamma (j+ \alpha+1)}{j! \Gamma (\alpha +1)} \overline{a}^j z^j,$$
$$h_a (z) = (1-|a|^2)^4 \sum_{j=0}^{\infty} \frac{\Gamma (j+ \alpha+2)}{j! \Gamma (\alpha +2)} \overline{a}^j z^j,$$
for each $a, z\in \mathbb{D}$. Note that
\begin{align*}
f_a (z) = & (1-|a|^2)^2 + (1-|a|^2)^2 \sum_{j=1}^{\infty} \frac{\Gamma (j+ \alpha)}{j! \Gamma (\alpha)} \overline{a}^j z^j \\
= & (1-|a|^2)^2 + (1-|a|^2)^2 \sum_{j=0}^{\infty} \frac{\Gamma (j+1+ \alpha)}{(j+1)! \Gamma (\alpha)} \overline{a}^{j+1} z^{j+1} \\
= & (1-|a|^2)^2 + (1-|a|^2)^2 \sum_{j=0}^{\infty} \frac{j+\alpha}{j+1} \frac{\Gamma (j+ \alpha)}{j! \Gamma (\alpha)} \overline{a}^{j+1} z^{j+1}.
\end{align*}
Also, by Stirling's formula, we know that $\frac{\Gamma (j+ \alpha)}{j! \Gamma (\alpha)}\approx j^{\alpha -1}$ as $j \rightarrow \infty$, see \cite{hu4}.
Consequently,
\begin{align*}
\| D_{\varphi,u}^n f_a\|_{\mathcal{Z}_{\beta}} \lesssim & (1-|a|^2)^2 \| D_{\varphi,u}^n 1\|_{\mathcal{Z}_{\beta}} + (1-|a|^2)^2 \sum_{j=0}^{\infty} \frac{j+\alpha}{j+1} j^{\alpha -1} |a| ^{j+1} \| D_{\varphi,u}^n I^{j+1}\|_{\mathcal{Z}_{\beta}} \\
= & (1-|a|^2)^2 \sum_{j=0}^{\infty} \frac{j+\alpha}{j+1} j |a| ^{j+1} j^{\alpha -2} \| D_{\varphi,u}^n I^{j+1}\|_{\mathcal{Z}_{\beta}} \\
\leq & (1-|a|^2)^2 \sum_{j=0}^{\infty} (j+\alpha) |a| ^{j+1} j^{\alpha -2} \| D_{\varphi,u}^n I^{j+1}\|_{\mathcal{Z}_{\beta}} \lesssim M,
\end{align*}
and since $a\in\mathbb{D}$ was arbitrary, we get $\sup_{a \in \mathbb{D}}\| D_{\varphi,u}^n f_a\|_{\mathcal{Z}_{\beta}} < \infty$.
Applying a similar argument to the functions $g_a$ and $h_a$ implies $(iii)$.\\
$(iii) \Rightarrow (iv)$ For each $a, z \in \mathbb{D}$ define
$$k_a (z) = (\alpha +n)\frac{(1-|a|^2)^2}{(1-\overline{a}z)^{\alpha}} -2 \alpha
\frac{(1-|a|^2)^3}{(1-\overline{a}z)^{\alpha+1}} + \frac{\alpha(\alpha +1)}{\alpha + n +1}
\frac{(1-|a|^2)^4}{(1-\overline{a}z)^{\alpha +2}}.$$
Then, one can see that $k_a \in \mathcal{Z}_{\alpha}$, $\sup_{a \in \mathbb{D}}
\| k_a \|_{\mathcal{Z}_{\alpha}} < \infty $,
${k^{(n)}_a} (a) = {k^{(n+1)}_a} (a) =0$ and
$$ {k^{(n+2)}_{\varphi(a)}}(\varphi(a)) = 2 \alpha (\alpha +1) \cdots
(\alpha +n) \frac{\overline{\varphi(a)}^{n+2}}{(1-|\varphi(a)|^2)^{\alpha +n}}.   $$
Moreover,
\begin{align*}
 \sup_{a \in \mathbb{D}} \| D_{\varphi,u}^n k_a\|_{\mathcal{Z}_{\beta}} \leq & (\alpha +n) \sup_{a \in \mathbb{D}} \| D_{\varphi,u}^n f_a\|_{\mathcal{Z}_{\beta}} + 2 \alpha \sup_{a \in \mathbb{D}} \| D_{\varphi,u}^n g_a\|_{\mathcal{Z}_{\beta}} \\
  & + \frac{\alpha (\alpha +1)}{\alpha +n+1}
\sup_{a \in \mathbb{D}} \| D_{\varphi,u}^n h_a\|_{\mathcal{Z}_{\beta}} < \infty,
\end{align*}
and by the definition of $\|\cdot \|_{\mathcal{Z}_{\beta}}$ we have
\begin{align*}
\infty >  \| D_{\varphi,u}^n k_{\varphi(a)}\|_{\mathcal{Z}_{\beta}}
\geq & \sup_{z \in \mathbb{D}}(1-|z|^2)^{\beta} | (u {k^{(n)}_{\varphi(a)}} \circ \varphi)''(z) | \\
\geq & (1-|a|^2)^{\beta}| (u {k^{(n)}_{\varphi(a)}} \circ \varphi)''(a) | \\
= & 2 \alpha (\alpha +1) \cdots (\alpha +n) (1-|a|^2)^{\beta} \frac{|u(a) \varphi'^2(a)| \ |\varphi(a)|^{n+2}}
{(1-|\varphi(a)|^2)^{\alpha +n}}.
\end{align*}
Therefore,
\begin{equation} \label{eq12}
\sup_{|\varphi(a)| > 1/2} \frac{ (1-|a|^2)^{\beta}}{(1-|\varphi(a)|^2)^{\alpha +n}} |u(a) \varphi'^2(a)| < \infty.
\end{equation}
On the other hand,
\begin{equation} \label{eq13}
\sup_{|\varphi(a)| \leq  1/2} \frac{(1-|a|^2)^{\beta}}{(1-|\varphi(a)|^2)^{\alpha +n}} |u(a) \varphi'^2(a)| \lesssim \sup_{|\varphi(a)| \leq  1/2} (1-|a|^2)^{\beta} |u(a) \varphi'^2(a)| < \infty.
\end{equation}
Now, \eqref{eq12} and \eqref{eq13} imply that
$$C(u,\varphi, \alpha, \beta, n) = \sup_{z \in \mathbb{D}} \frac{ (1-|z|^2)^{\beta} }{(1-|\varphi(z)|^2)^{\alpha +n}} |u(z) \varphi'^2(z)| < \infty.$$
In order to prove that $B(u,\varphi, \alpha, \beta, n) < \infty$,  the similar approach can be applied using  the test functions $l_a$ defined as
$$ l_a (z) = (\alpha +n)(\alpha +n+2)\frac{(1-|a|^2)^2}{(1-\overline{a}z)^{\alpha}} - \alpha (2 \alpha +2n +3)
\frac{(1-|a|^2)^3}{(1-\overline{a}z)^{\alpha+1}} + \alpha (\alpha +1)
\frac{(1-|a|^2)^4}{(1-\overline{a}z)^{\alpha +2}}.  $$
Also, for the proof of $A(u,\varphi, \alpha, \beta, n)<\infty$ the argument is similar, using the test functions
$$ m_a (z) = (\alpha +n+1)\frac{(1-|a|^2)^2}{(1-\overline{a}z)^{\alpha}} - 2\alpha
\frac{(1-|a|^2)^3}{(1-\overline{a}z)^{\alpha+1}} + \frac{\alpha (\alpha +1)}{\alpha+n +2}
\frac{(1-|a|^2)^4}{(1-\overline{a}z)^{\alpha +2}}.$$
The special case of $n=1$ and $0<\alpha <1$ can be proved by a similar argument and applying Lemma \ref{l31}.\\
$(iv) \Rightarrow (i)$
By applying \eqref{eq11}, for every $f \in \mathcal{Z}_{\alpha} $ and $n \geq 2$ we get
\begin{align*}
| & (D_{\varphi,u}^n f)''(z)| \\
& =  |u''(z) f^{(n)}(\varphi(z)) + (2 u'(z) \varphi'(z) + u(z) \varphi''(z))f^{(n+1)}(z)+u(z) \varphi'^2 (z) f^{(n+2)}(\varphi(z))| \\
& \leq |u''(z) f^{(n)}(\varphi(z))|  +  |(2 u'(z) \varphi'(z) + u(z) \varphi''(z))f^{(n+1)}(z)| + |u(z) \varphi'^2 (z) f^{(n+2)}(\varphi(z))| \\
& \leq  \frac{ |u''(z)|}{(1-|\varphi(z)|^2)^{\alpha +n-2}} \|f\|_{\mathcal{Z}_{\alpha}} +  \frac{|2 u'(z) \varphi'(z) + u(z) \varphi''(z)|}{(1-|\varphi(z)|^2)^{\alpha +n-1}} \|f\|_{\mathcal{Z}_{\alpha}} +  \frac{|u(z) \varphi'^2(z)|}{(1-|\varphi(z)|^2)^{\alpha + n}} \|f\|_{\mathcal{Z}_{\alpha}}.
\end{align*}
Multiplying both sides by $(1-|z|^2)^{\beta}$ and taking supremum over $z \in \mathbb{D}$, one can get $(i)$.
\end{proof}
In the case of $n =1$ and $\alpha =1$ we have the following result.
\begin{theorem} \label{g35}
For each $0<\beta <\infty$, $D_{\varphi,u}^1: \mathcal{Z} \rightarrow \mathcal{Z}_{\beta}$ is bounded if and only if
\begin{itemize}
\item[$(i)$] $ \sup_{z \in \mathbb{D}} (1-|z|^2)^{\beta} |u''(z)| \log \frac{1}{1-|\varphi(z)|^2} < \infty$,
\item[$(ii)$] $\sup_{z \in \mathbb{D}} \frac{ (1-|z|^2)^{\beta} }{1-|\varphi(z)|^2} |2 u'(z) \varphi'(z) + u(z) \varphi''(z)| < \infty$,
\item[$(iii)$] $\sup_{z \in \mathbb{D}} \frac{ (1-|z|^2)^{\beta} }{(1-|\varphi(z)|^2)^2} |u(z) \varphi'^2(z)| < \infty$.
\end{itemize}
\end{theorem}
\begin{proof}
The sufficiency part is a consequence of Lemma \ref{l31}$(ii)$ and  definition of the norm in Zygmund type spaces. Now, suppose that $ D_{\varphi,u}^1 : \mathcal{Z} \rightarrow \mathcal{Z}_{\beta}$ is bounded.
Then, by applying $D_{\varphi,u}^1 z, D_{\varphi,u}^1 z^2, D_{\varphi,u}^1 z^3 \in \mathcal{Z}_{\beta}$ we get
$u, (u' \varphi' + u \varphi''), u \varphi'^2 \in \mathcal{Z}_{\beta}$.
Also, by defining
$$ k_a (z) =  3 \frac{(1-|a|^2)^2}{1-\overline{a}z} -3 \frac{(1-|a|^2)^3}{(1-\overline{a}z)^{2}} + \frac{(1-|a|^2)^4}{(1-\overline{a}z)^{3}},$$
for each $a, z \in \mathbb{D}$, one can see that  $k_{a} \in \mathcal{Z}$, $\sup_{a\in \mathbb{D}}\| k_{a}\|_\mathcal{Z}<\infty$, $k_{\varphi(a)}'(\varphi(a)) =0 $,  $k_{\varphi(a)}''(\varphi(a)) =0$ and
$ k_{\varphi(a)}'''(\varphi(a)) =  16 \frac{ \overline{\varphi(a)}^3}{(1-|\varphi(a)|^2)^{2}}$. Therefore,
\begin{align*}
\|D_{\varphi,u}^1 k_{\varphi(a)}\|_{\mathcal{Z}_{\beta}} \geq &  (1-|a|^2)^{\beta}
|u(a) \varphi'^2 (a) f'''(\varphi(a))| \\
= & 16 \frac{(1-|a|^2)^{\beta}}{(1-|\varphi(a)|^2)^{3}} |u(a) \varphi'^2(a)| |\varphi(a)|^3,
\end{align*}
and consequently,
\begin{equation*}
\sup_{|\varphi(a)| > 1/2} \frac{(1-|a|^2)^{\beta}}{(1-|\varphi(a)|^2)^{2}} |u(a) \varphi'^2(a)| \lesssim \sup_{|\varphi(a)| > 1/2} \|D_{\varphi,u}^1 k_{\varphi(a)}\|_{\mathcal{Z}_{\beta}} < \infty.
\end{equation*}
On the other hand, $u \varphi'^2 \in \mathcal{Z}_{\beta}$ implies that
\begin{equation*}
\sup_{|\varphi(a)| \leq  1/2} \frac{(1-|a|^2)^{\beta}}{(1-|\varphi(a)|^2)^{2}} |u(a) \varphi'^2(a)|  < \infty,
\end{equation*}
which completes the proof of $(iii)$. In order to prove $(ii)$, for each $a, z \in \mathbb{D}$, define the test functions
\begin{align*}
l_a (z) =  8 \frac{(1-|a|^2)^2}{1-\overline{a}z} - 7 \frac{(1-|a|^2)^3}{(1-\overline{a}z)^{2}} + 2  \frac{(1-|a|^2)^4}{(1-\overline{a}z)^{3}}.
\end{align*}
Then, one can prove $(ii)$ by a similar approach as in $(iii)$ and using the facts  $l_{a} \in \mathcal{Z}$, $\sup_{a\in \mathbb{D}}\| l_{a}\|_\mathcal{Z}<\infty$,
$l_{\varphi(a)}'(\varphi(a)) =0 $,  $l_{\varphi(a)}'''(\varphi(a)) =0$ and
$ l_{\varphi(a)}''(\varphi(a)) =  - 2\frac{ \overline{\varphi(a)}^2}{(1-|\varphi(a)|^2)^{\alpha}}$.

In order to prove $(i)$, consider the test functions
$$t_a(z) = \frac{h(\overline{\varphi(a)}z)}{\overline{\varphi(a)}} \left ( \log \frac{1}{1-|\varphi(a)|^2} \right)^{-1},$$
for each $a, z \in \mathbb{D}$, where $h(z) = (z-1) \left( ( 1+ \log \frac{1}{1-z})^2 +1 \right)$. Then, one can see that $t_{a} \in \mathcal{Z}$,  $\sup_{a \in \mathbb{D}} \|t_a\|_{\mathcal{Z}} < \infty$, $t_a'(\varphi(a)) = \log \frac{1}{1-|\varphi(a)|^2}$ and
$$t_a''(\varphi(a)) = \frac{2 \overline{\varphi(a)}}{1-|\varphi(a)|^2}, \quad
t_a'''(\varphi(a)) = \frac{2 \overline{\varphi(a)}^2}{(1-|\varphi(a)|^2)^2} (1+ (\log \frac{1}{1-|\varphi(a)|^2})^{-1}). $$
Since the operator $D_{\varphi,u}^1: \mathcal{Z} \rightarrow \mathcal{Z}_{\beta}$ is bounded, we get
\begin{align*}
\|D_{\varphi,u}^1 t_a\|_{\mathcal{Z}_{\beta}} \geq &  (1-|a|^2)^{\beta} |u''(a) t_a'(\varphi(a))| \\
& - (1-|a|^2)^{\beta} |(2 u'(a) \varphi'(a) + u(a) \varphi''(a)) t_a'' (\varphi(a))| -  |u(a) \varphi'^2 (a) t_a''' (\varphi(a))|\\
=&  (1-|a|^2)^{\beta} |u''(a)| \log \frac{1}{1-|\varphi(a)|^2} \\
&  - (1-|a|^2)^{\beta}  |2 u'(a) \varphi'(a) + u(a) \varphi''(a)| \frac{2 \overline{\varphi(a)}}{1-|\varphi(a)|^2} \\
& - (1-|a|^2)^{\beta}|u(a) \varphi'^2 (a) | \frac{2 \overline{\varphi(a)}^2}{(1-|\varphi(a)|^2)^2} (1+ (\log \frac{1}{1-|\varphi(a)|^2})^{-1}).
\end{align*}
Therefore,
\begin{align*}
\sup_{|\varphi(a)| > 1/2} & (1-|a|^2)^{\beta}  |u''(a)| \log \frac{1}{1-|\varphi(a)|^2}   \leq  \sup_{|\varphi(a)| > 1/2} \|D_{\varphi,u}^1 t_a\|_{\mathcal{Z}_{\beta}} \\
+ & \sup_{|\varphi(a)| > 1/2} (1-|a|^2)^{\beta}  |2 u'(a) \varphi'(a) + u(a) \varphi''(a)| \frac{2 \overline{\varphi(a)}}{1-|\varphi(a)|^2} \\
+ & \sup_{|\varphi(a)| > 1/2} (1-|a|^2)^{\beta} |u(a) \varphi'^2 (a) | \frac{2 \overline{\varphi(a)}^2}{(1-|\varphi(a)|^2)^2} (1+ (\log \frac{1}{1-|\varphi(a)|^2})^{-1})\\
< & \infty.
\end{align*}
On the other hand, since $u \in \mathcal{Z}_{\beta}$, we have
$$ \sup_{|\varphi(a)| \leq 1/2} (1-|a|^2)^{\beta} |u''(a)| \log \frac{2}{1-|\varphi(a)|^2} < \infty,$$
which completes the proof.
\end{proof}
\section{\sc\bf Essential norms}
In this section we investigate essential norm estimates of $D_{ \varphi, u}^n: \mathcal{Z}_{\alpha} \rightarrow \mathcal{Z}_{\beta}$ in different cases of $\alpha$, $\beta$ and $n$.
In order to simplify the notations, we define
$$ E = \limsup_{|a| \rightarrow 1} \|D_{ \varphi, u}^n f_a\|_{\mathcal{Z}_{\beta}} , \quad F = \limsup_{|a| \rightarrow 1} \|D_{ \varphi, u}^n g_a\|_{\mathcal{Z}_{\beta}}, \quad G = \limsup_{|a| \rightarrow 1} \|D_{ \varphi, u}^n h_a\|_{\mathcal{Z}_{\beta}}, $$
and
\begin{align*}
A = & \limsup_{|\varphi (z)| \rightarrow 1} \frac{(1-|z|^2)^{\beta}}{(1-|\varphi(z)|^2)^{\alpha + n-2}} |u''(z)|, \\
B = & \limsup_{|\varphi (z)| \rightarrow 1} \frac{ (1-|z|^2)^{\beta} }{(1-|\varphi(z)|^2)^{\alpha + n -1}} |2u'(z) \varphi'(z) + u(z) \varphi''(z) |, \\
C = & \limsup_{|\varphi (z)| \rightarrow 1} \frac{ (1-|z|^2)^{\beta} }{(1-|\varphi(z)|^2)^{\alpha + n}} |u(z) \varphi'^2(z)|.
\end{align*}
\begin{theorem} \label{g38}
Let $0<\alpha, \beta < \infty $ and $n \geq 1$ with $(n, \alpha)\neq (1, 1)$. If $D_{ \varphi, u}^n: \mathcal{Z}_{\alpha} \rightarrow \mathcal{Z}_{\beta}$ is a bounded operator, then
$$ \|D_{ \varphi, u}^n\|_{e,\mathcal{Z}_{\alpha} \rightarrow  \mathcal{Z}_{\beta}} \approx \limsup_{j \rightarrow \infty} j^{\alpha -2} \| D_{\varphi,u}^n I^{j+1}\|_{\mathcal{Z}_{\beta}} \approx \max \{ E, F, G \} \approx \max \{ A, B , C \}. $$
\end{theorem}
\begin{proof}
First note that, as mentioned in the proof of Theorem \ref{g34}, $ \{{j^{\alpha -2} I^{j+1} }\}$ is a bounded sequence in $\mathcal{Z}_{\alpha}$. Since polynomials are dense in $\mathcal{Z}_{\alpha,0}$, by  \cite[Proposition 2.1]{cown}, $ \{{j^{\alpha -2}} I^{j+1} \}$ converges weakly to $0$ in $\mathcal{Z}_{\alpha,0}$ and hence in  $\mathcal{Z}_{\alpha}$. Therefore, for each compact operator $K : \mathcal{Z}_{\alpha} \rightarrow \mathcal{Z}_{\beta} $, we have $\lim_{j \rightarrow \infty} \| K(j^{\alpha -2} I^{j+1}) \|_{\mathcal{Z}_{\beta}} =0$ and consequently
\begin{align} \label{eq14}
\|D_{ \varphi, u}^n\|_{e,\mathcal{Z}_{\alpha} \rightarrow  \mathcal{Z}_{\beta}} = & \inf_{K} \|D_{ \varphi, u}^n - K\|_{\mathcal{Z}_{\alpha} \rightarrow  \mathcal{Z}_{\beta}} \nonumber \\
\gtrsim & \inf_{K} \limsup_{j \rightarrow \infty}\|(D_{ \varphi, u}^n - K)(j^{\alpha -2} I^{j+1})\|_{\mathcal{Z}_{\beta}} \nonumber \\
\gtrsim & \limsup_{j \rightarrow \infty} j^{\alpha -2} \| D_{ \varphi, u}^n I^{j+1}\|_{\mathcal{Z}_{\beta}}.
\end{align}
Next we prove that
$$ \|D_{ \varphi, u}^n\|_{e,\mathcal{Z}_{\alpha} \rightarrow  \mathcal{Z}_{\beta}} \approx \max \{ E, F, G \}. $$
One can see that $f_a, g_a, h_a \in \mathcal{Z}_{\alpha,0}$ having uniformly bounded $\|\cdot \|_{\mathcal{Z}_{\alpha}}$ norm. Also, these functions converge to $0$ uniformly on compact subsets of $\mathbb{D}$, as $|a| \rightarrow 1$, which implies their weak convergence to $0$ in $\mathcal{Z}_{\alpha}$. Therefore, for each compact operator $K : \mathcal{Z}_{\alpha} \rightarrow \mathcal{Z}_{\beta} $, we have
\begin{align*}
 \| D_{\varphi, u}^n - K \|_{\mathcal{Z}_{\alpha} \rightarrow \mathcal{Z}_{\beta}} \gtrsim &  \limsup_{|a| \rightarrow 1} \| (D_{\varphi, u}^n - K)f_a \|_{\mathcal{Z}_{\beta}} \\
\geq  & \limsup_{|a| \rightarrow 1} \| D_{\varphi, u}^n f_a\|_{\mathcal{Z}_{\beta}} - \limsup_{|a| \rightarrow 1} \| K f_a\|_{\mathcal{Z}_{\beta}} \\
= & \limsup_{|a| \rightarrow 1} \| D_{\varphi, u}^n f_a\|_{\mathcal{Z}_{\beta}} = E.
\end{align*}
The same argument is valid for $g_a$ and $h_a$, that is
$$  \| D_{\varphi, u}^n - K \|_{\mathcal{Z}_{\alpha} \rightarrow \mathcal{Z}_{\beta}} \gtrsim F, \quad  \| D_{\varphi, u}^n - K \|_{\mathcal{Z}_{\alpha} \rightarrow \mathcal{Z}_{\beta}} \gtrsim G.$$
Since $K : \mathcal{Z}_{\alpha} \rightarrow \mathcal{Z}_{\beta} $ was an arbitrary compact operator, we get
\begin{equation*} \label{eq15}
\| D_{\varphi, u}^n\|_{e, \mathcal{Z}_{\alpha} \rightarrow \mathcal{Z}_{\beta}}  = \inf_{K} \| D_{\varphi, u}^n - K \|_{\mathcal{Z}_{\alpha} \rightarrow \mathcal{Z}_{\beta}} \gtrsim \max \{ E, F, G \}.
\end{equation*}
For each $0 < r <1$ and $f\in H(\mathbb{D})$ the function $f_r\in H(\mathbb{D})$ is defined by $f_r (z) = f(rz)$ for all $z\in \mathbb{D}$. Note that $f_r \rightarrow f$ uniformly on compact subsets of $\mathbb{D}$ when $r \rightarrow 1$. Also, the operator $K_r:\mathcal{Z}_{\alpha} \rightarrow \mathcal{Z}_{\alpha}$ given by $K_r f = f_r$ is a compact operator for each $0 < r <1$.
Let $\{r_j\}$ be a sequence in $(0,1)$ converging to $1$ as $j \rightarrow \infty$.
Then, $D_{\varphi, u}^n K_{r_j} :\mathcal{Z}_{\alpha} \rightarrow \mathcal{Z}_{\beta} $ is a compact operator for each $j\geq 1$.  Hence,
$$ \| D_{\varphi, u}^n\|_{e, \mathcal{Z}_{\alpha} \rightarrow \mathcal{Z}_{\beta}} \leq \limsup_{j \rightarrow \infty}
\| D_{\varphi, u}^n - D_{\varphi, u}^n K_{r_j}\|_{\mathcal{Z}_{\alpha} \rightarrow \mathcal{Z}_{\beta}}. $$
Therefore, it is enough to prove that
$$  \limsup_{j \rightarrow \infty}
\| D_{\varphi, u}^n - D_{\varphi, u}^n K_{r_j}\|_{\mathcal{Z}_{\alpha} \rightarrow \mathcal{Z}_{\beta}} \lesssim \max \{ E, F, G\}. $$
Note that for every $f \in \mathcal{Z}_{\alpha}$ with $\|f\|_{\mathcal{Z}_{\alpha}} \leq 1$ we have
\begin{align*}
\| & (D_{\varphi, u}^n - D_{\varphi, u}^n K_{r_j})f\|_{\mathcal{Z}_{\beta}}   \\
& = |u(0)f^{(n)} (\varphi(0)) - r_j^n  u(0)f^{(n)} (r_j\varphi(0)) | \\
& \ \ + |u'(0)(f-f_{r_j})^{(n)} (\varphi(0)) + u(0)(f-f_{r_j})^{(n+1)} (\varphi(0))\varphi'(0)| \\
& \ \ + \|u\cdot (f-f_{r_j})^{(n)} \circ \varphi\|_{s\mathcal{Z}_{\beta}} .
\end{align*}
Since $r_j \rightarrow 1$, we have 
$$ \limsup_{j \rightarrow \infty} |u(0)f^{(n)} (\varphi(0)) - r_j^n  u(0)f^{(n)} (r_j\varphi(0)) | =0,$$
and
$$ \limsup_{j \rightarrow \infty} |u(0)(f-f_{r_j})^{(n)} (\varphi(0)) + u(0)(f-f_{r_j})^{(n+1)} (\varphi(0))\varphi'(0)| =0. $$
Therefore, for each $N\geq 1$, we get 
\begin{align*}
\limsup_{j \rightarrow \infty} & \|  (D_{\varphi, u}^n - D_{\varphi, u}^n K_{r_j})f\|_{\mathcal{Z}_{\beta}} \leq
\limsup_{j \rightarrow \infty} \|u\cdot (f-f_{r_j})^{(n)} \circ \varphi\|_{s\mathcal{Z}_{\beta}} \\
\leq & \limsup_{j \rightarrow \infty}  \sup_{|\varphi(z)| \leq r_N} (1-|z|^2)^{\beta} |(f - f_{r_j})^{(n)} (\varphi(z))| |u''(z)| \\
& + \limsup_{j \rightarrow \infty}  \sup_{|\varphi(z)| > r_N} (1-|z|^2)^{\beta} |(f - f_{r_j})^{(n)} (\varphi(z))| |u''(z)| \\
& + \limsup_{j \rightarrow \infty}  \sup_{|\varphi(z)| \leq r_N} (1-|z|^2)^{\beta} |(f - f_{r_j})^{(n +1)} (\varphi(z))| |2u'(z) \varphi'(z) + u(z) \varphi''(z)| \\
& + \limsup_{j \rightarrow \infty}  \sup_{|\varphi(z)| > r_N} (1-|z|^2)^{\beta} |(f - f_{r_j})^{(n +1)} (\varphi(z))| |2u'(z) \varphi'(z) + u(z) \varphi''(z)| \\
& + \limsup_{j \rightarrow \infty}  \sup_{|\varphi(z)| \leq r_N} (1-|z|^2)^{\beta} |(f - f_{r_j})^{(n +2)} (\varphi(z))| | u(z) \varphi'^2(z)| \\
& + \limsup_{j \rightarrow \infty}  \sup_{|\varphi(z)| > r_N} (1-|z|^2)^{\beta} |(f - f_{r_j})^{(n +2)} (\varphi(z))| | u(z) \varphi'^2(z)| \\
:= & S_1 + S_2 + S_3 + S_4 + S_5 + S_6,
\end{align*}
respectively. Since the operator $D_{ \varphi, u}^n: \mathcal{Z}_{\alpha} \rightarrow \mathcal{Z}_{\beta}$ is bounded, by Theorem \ref{g34}, we have
$$ \sup_{z \in \mathbb{D}}  (1-|z|^2)^{\beta} |u''(z)| < \infty, \quad  \sup_{z \in \mathbb{D}}  (1-|z|^2)^{\beta} |2u'(z) \varphi'(z) + u(z) \varphi''(z)| < \infty,$$
$$\sup_{z \in \mathbb{D}}  (1-|z|^2)^{\beta} |u(z) \varphi'^2(z) | < \infty.$$
Also, for each $k\geq 1$, $f_{r_j}^{(k)} \rightarrow f^{(k)}$ uniformly on compact subsets of $\mathbb{D}$, since $f_{r_j} \rightarrow f$ uniformly on compact subsets of $\mathbb{D}$. These facts imply that
$$S_1 = S_3 = S_5 =0,$$
and hence we just need to estimate $S_2$, $S_4$ and $S_6$. By applying \eqref{eq11}, Lemma \ref{l31} and using the test functions $m_a$ defined in Theorem \ref{g34}, we have
\begin{align} \label{eq16}
S_2 = & \limsup_{j \rightarrow \infty}  \sup_{|\varphi(z)| > r_N} (1-|z|^2)^{\beta} |(f - f_{r_j})^{(n)} (\varphi(z))| |u''(z)| \nonumber \\
\leq & \limsup_{j \rightarrow \infty}  \sup_{|\varphi(z)| > r_N} (1-|z|^2)^{\beta} \frac{\|f -f_{r_j}\|_{\mathcal{Z}_{\alpha}}}{(1-|\varphi(z)|^2)^{\alpha + n -2}} |u''(z)| \nonumber  \\
\lesssim & \sup_{|\varphi(z)| > r_N} (1-|z|^2)^{\beta} \frac{|\varphi(z)|^n}{|\varphi(z)|^n(1-|\varphi(z)|^2)^{\alpha + n -2}} |u''(z)| \\
\lesssim  & \sup_{|a| > r_N} \frac{\alpha + n+2}{2 \alpha (\alpha +1) \cdots (\alpha + n -1)} \|D_{ \varphi, u}^n (m_a)\|_{\mathcal{Z}_{\beta}} \nonumber \\
\lesssim & \sup_{|a| > r_N} \|D_{ \varphi, u}^n \left( (\alpha +n+1)f_a - 2\alpha g_a + \frac{\alpha (\alpha +1)}{\alpha+n +2} h_a \right)\|_{\mathcal{Z}_{\beta}} \nonumber \\
\leq  & (\alpha +n+1) \sup_{|a| > r_N} \|D_{ \varphi, u}^n f_a\|_{\mathcal{Z}_{\beta}} + 2 \alpha  \sup_{|a| > r_N} \|D_{ \varphi, u}^n g_a\|_{\mathcal{Z}_{\beta}} \nonumber \\
&  + \frac{\alpha (\alpha +1)}{\alpha+n +2} \sup_{|a| > r_N} \|D_{ \varphi, u}^n h_a\|_{\mathcal{Z}_{\beta}}. \nonumber
\end{align}
Therefore, as $N \rightarrow \infty$, we get
\begin{align*}
S_2 \lesssim &  \limsup_{|a| \rightarrow 1} \|D_{ \varphi, u}^n f_a\|_{\mathcal{Z}_{\beta}} + \limsup_{|a| \rightarrow 1} \|D_{ \varphi, u}^n g_a\|_{\mathcal{Z}_{\beta}} + \limsup_{|a| \rightarrow 1} \|D_{ \varphi, u}^n h_a\|_{\mathcal{Z}_{\beta}} \\
\lesssim & \max \{ E, F, G \}.
\end{align*}
Similarly, by applying the test functions $k_a$ and $l_a$, one can prove 
$$S_4 \lesssim \max \{ E, F, G \}, \quad S_6 \lesssim \max \{ E, F, G \},$$
and therefore
\begin{equation} \label{eq17}
\| D_{\varphi, u}^n\|_{e, \mathcal{Z}_{\alpha} \rightarrow \mathcal{Z}_{\beta}} \leq \limsup_{j \rightarrow \infty}
\| D_{\varphi, u}^n - D_{\varphi, u}^n K_{r_j}\|_{\mathcal{Z}_{\alpha} \rightarrow \mathcal{Z}_{\beta}} \lesssim \max \{ E, F, G \}.
\end{equation}
By applying \eqref{eq16}, we get
$$ S_2 \lesssim 	  \limsup_{|\varphi(z)| \rightarrow 1}  \frac{(1-|z|^2)^{\beta}}{(1-|\varphi(z)|^2)^{\alpha + n -2}} |u''(z)| =A, $$
and similarly one can see that $S_4 \lesssim B$ and $S_6 \lesssim C$.
Therefore,
\begin{equation} \label{eq18}
\| D_{\varphi, u}^n\|_{e, \mathcal{Z}_{\alpha} \rightarrow \mathcal{Z}_{\beta}} \lesssim  \max \{ A, B, C \}.
\end{equation}
In order to prove the converse of \eqref{eq18}, let $\{ z_j \}$ be a sequence in $\mathbb{D}$ such that
$|\varphi(z_j)| \rightarrow 1$, as $j \rightarrow \infty$, and define the test functions
$$ k_j = k_{\varphi(z_j)}, \quad l_j = l_{\varphi(z_j)}, \quad m_j = m_{\varphi(z_j)},$$
where $k_a$, $l_a$ and $m_a$ are defined in the proof of Theorem \ref{g34}. Then, $\{k_j\}, \{l_j\}$ and $\{m_j\}$ are bounded sequences in $\mathcal{Z}_{\alpha,0}$ which converge to zero on compact subsets of $\mathbb{D}$. Also, for each $j\geq 1$ we have
\begin{align*}
{k^{(n)}_j} (\varphi(z_j)) & = {k^{(n+1)}_j} (\varphi(z_j)) =0,\\
{k^{(n+2)}_{j}}(\varphi(z_j)) & = 2 \alpha (\alpha +1) \cdots (\alpha +n) \frac{|\varphi(z_j)|^{n+2}}{(1-|\varphi(z_j)|^2)^{\alpha +n}},\\
{l^{(n)}_j} (\varphi(z_j)) & = {l^{(n+2)}_j} (\varphi(z_j)) =0,\\
{l^{(n+1)}_{j}}(\varphi(z_j)) & = -\alpha (\alpha +1) \cdots (\alpha +n) \frac{|\varphi(z_j)|^{n+1}}{(1-|\varphi(z_j)|^2)^{\alpha + n -1}},\\
{m^{(n+1)}_j} (\varphi(z_j)) & = {m^{(n+2)}_j} (\varphi(z_j)) =0,\\
{m^{(n)}_{j}}(\varphi(z_j)) & = \frac{\alpha + n +2}{2 \alpha (\alpha +1) \cdots (\alpha +n -1)} \frac{|\varphi(z_j)|^{n}}{(1-|\varphi(z_j)|^2)^{\alpha + n -2}}.
\end{align*}
Thus, for any compact operator $K: \mathcal{Z}_{\alpha} \rightarrow \mathcal{Z}_{\beta}$ we get 
\begin{align*}
\| D_{\varphi, u}^n - K\|_{\mathcal{Z}_{\alpha} \rightarrow \mathcal{Z}_{\beta}} \gtrsim &  \limsup_{j \rightarrow \infty}
\| D_{\varphi, u}^n (k_j)\|_{\mathcal{Z}_{\beta}} - \limsup_{j \rightarrow \infty} \| K (k_j)\|_{\mathcal{Z}_{\beta}} \\
\gtrsim   & \limsup_{j \rightarrow \infty}  \frac{(1-|z_j|^2)^{\beta} |\varphi(z_j)|^{n+2}}{(1-|\varphi(z_j)|^2)^{\alpha +n}} |u(z_j)| |\varphi'^2(z_j)| \\
=  & \limsup_{|\varphi(z)| \rightarrow 1}  \frac{(1-|z|^2)^{\beta}}{(1-|\varphi(z)|^2)^{\alpha +n}} |u(z)| |\varphi'^2(z)|,
\end{align*}
\begin{align*}
\| D_{\varphi, u}^n - K\|_{\mathcal{Z}_{\alpha} \rightarrow \mathcal{Z}_{\beta}} \gtrsim &  \limsup_{j \rightarrow \infty}
\| D_{\varphi, u}^n (l_j)\|_{\mathcal{Z}_{\beta}} - \limsup_{j \rightarrow \infty} \| K (l_j)\|_{\mathcal{Z}_{\beta}} \\
\gtrsim  & \limsup_{j \rightarrow \infty}  \frac{(1-|z_j|^2)^{\beta} |\varphi(z_j)|^{n+1}}{(1-|\varphi(z_j)|^2)^{\alpha +n -1}} |2 u'(z_j) \varphi'(z_j) + u(z_j)\varphi''(z_j)| \\
=  & \limsup_{|\varphi(z)| \rightarrow 1} \frac{(1-|z|^2)^{\beta}}{(1-|\varphi(z)|^2)^{\alpha +n -1}} |2 u'(z) \varphi'(z) + u(z)\varphi''(z)|,
\end{align*}
and
\begin{align*}
\| D_{\varphi, u}^n - K\|_{\mathcal{Z}_{\alpha} \rightarrow \mathcal{Z}_{\beta}} \gtrsim &  \limsup_{j \rightarrow \infty}
\| D_{\varphi, u}^n (m_j)\|_{\mathcal{Z}_{\beta}} - \limsup_{j \rightarrow \infty} \| K (m_j)\|_{\mathcal{Z}_{\beta}} \\
\gtrsim  & \limsup_{j \rightarrow \infty}  \frac{(1-|z_j|^2)^{\beta} |\varphi(z_j)|^{n}}{(1-|\varphi(z_j)|^2)^{\alpha +n -2}} |u''(z_j)| \\
=  & \limsup_{|\varphi(z)| \rightarrow 1} \frac{(1-|z|^2)^{\beta}}{(1-|\varphi(z)|^2)^{\alpha +n -2}} |u''(z)|.
\end{align*}
Hence,
\begin{align*}
\| D_{\varphi, u}^n\|_{e, \mathcal{Z}_{\alpha} \rightarrow \mathcal{Z}_{\beta}}  = & \inf_{K} \| D_{\varphi, u}^n - K\|_{\mathcal{Z}_{\alpha} \rightarrow \mathcal{Z}_{\beta}} \\
\gtrsim & \limsup_{|\varphi(z)| \rightarrow 1}  \frac{(1-|z|^2)^{\beta}}{(1-|\varphi(z)|^2)^{\alpha +n}} |u(z)| |\varphi'^2(z)| = C,
\end{align*}
and similarly
$$\| D_{\varphi, u}^n\|_{e, \mathcal{Z}_{\alpha} \rightarrow \mathcal{Z}_{\beta}} \gtrsim B, \quad \| D_{\varphi, u}^n\|_{e, \mathcal{Z}_{\alpha} \rightarrow \mathcal{Z}_{\beta}} \gtrsim A,$$
which implies that
\begin{equation*} \label{eq19}
\| D_{\varphi, u}^n\|_{e, \mathcal{Z}_{\alpha} \rightarrow \mathcal{Z}_{\beta}} \gtrsim \max \{A, B, C\}.
\end{equation*}
To prove the final case, first note that by Theorem \ref{g34} we have
$$M= \sup_{j \geq 1} j^{\alpha -2} \| D_{\varphi,u}^n I^{j+1}\|_{\mathcal{Z}_{\beta}} < \infty. $$
Recall that, for each $a, z \in \mathbb{D}$
$$ f_a (z) = \frac{(1-|a|^2)^2}{(1-\overline{a}z)^{\alpha}} = (1-|a|^2)^2 \sum_{j=0}^{\infty} \frac{\Gamma (j+ \alpha)}{j! \Gamma (\alpha)} \overline{a}^j z^j.$$
For each $N \geq 1$, as in the proof of Theorem \ref{g34} we have
\begin{align*}
\| & D_{\varphi, u}^n  f_a\|_{\mathcal{Z}_{\beta}} \lesssim (1-|a|^2)^{2} \sum_{j=0}^{\infty} (j+\alpha)|a|^{j+1} j^{\alpha -2} \| D_{\varphi, u}^n I^{j+1}\|_{\mathcal{Z}_{\beta}} \\
=  & (1-|a|^2)^{2} \left ( \sum_{j=0}^{N-1}  (j+\alpha)|a|^{j+1} j^{\alpha -2} \| D_{\varphi, u}^n I^{j+1}\|_{\mathcal{Z}_{\beta}} + \sum_{j=N}^{\infty}  (j+\alpha)|a|^{j+1} j^{\alpha -2} \| D_{\varphi, u}^n I^{j+1}\|_{\mathcal{Z}_{\beta}} \right ) \\
\lesssim  & (1-|a|^{2})^2 \sum_{j=0}^{N-1}  (j+\alpha)|a|^{j+1} + \sup_{j \geq N} j^{\alpha -2} \| D_{\varphi,u}^n I^{j+1}\|_{\mathcal{Z}_{\beta}} \sum_{j=N}^{\infty}  (j+\alpha)|a|^{j+1} ,
\end{align*}
which implies that
$$\limsup_{|a| \rightarrow 1}\| D_{\varphi, u}^n  f_a\|_{\mathcal{Z}_{\beta}}  \lesssim
 \sup_{j \geq N} j^{\alpha -2} \| D_{\varphi,u}^n I^{j+1}\|_{\mathcal{Z}_{\beta}}.$$
Therefore, when $N \rightarrow \infty$, we get
$$E = \limsup_{|a| \rightarrow 1}\| D_{\varphi, u}^n  f_a\|_{\mathcal{Z}_{\beta}}  \lesssim
\limsup_{j \rightarrow \infty} j^{\alpha -2} \| D_{\varphi,u}^n I^{j+1}\|_{\mathcal{Z}_{\beta}}.$$
A similar approach implies that
$$F \lesssim \limsup_{j \rightarrow \infty}j^{\alpha -2} \| D_{\varphi,u}^n I^{j+1}\|_{\mathcal{Z}_{\beta}}, \quad G \lesssim \limsup_{j \rightarrow \infty} j^{\alpha -2} \| D_{\varphi,u}^n I^{j+1}\|_{\mathcal{Z}_{\beta}},$$
and, by applying \eqref{eq17}, we conclude that
\begin{equation*} \label{eq20}
\| D_{\varphi, u}^n\|_{e, \mathcal{Z}_{\alpha} \rightarrow \mathcal{Z}_{\beta}} \lesssim \max \{ E, F, G \} \lesssim
\limsup_{j \rightarrow \infty}j^{\alpha -2} \| D_{\varphi,u}^n I^{j+1}\|_{\mathcal{Z}_{\beta}}.
\end{equation*}
This, along with \eqref{eq14}, completes the proof.
\end{proof}
\begin{cor} \label{cor39}
Let $0<\alpha, \beta < \infty $ and $n \geq 1$ with $(n, \alpha)\neq (1, 1)$. If $D_{ \varphi, u}^n: \mathcal{Z}_{\alpha} \rightarrow \mathcal{Z}_{\beta}$ is a bounded operator, then the following statements are equivalent:
\begin{itemize}
\item[$(i)$] $D_{ \varphi, u}^n: \mathcal{Z}_{\alpha} \rightarrow \mathcal{Z}_{\beta}$ is compact.
\item[$(ii)$] $\limsup_{j \rightarrow \infty} j^{\alpha -2} \| D_{\varphi,u}^n I^{j+1}\|_{\mathcal{Z}_{\beta}} =0 $.
\item[$(iii)$] $\limsup_{|a| \rightarrow 1} \|D_{ \varphi, u}^n f_a\|_{\mathcal{Z}_{\beta}}= \limsup_{|a| \rightarrow 1} \|D_{ \varphi, u}^n g_a\|_{\mathcal{Z}_{\beta}}= \limsup_{|a| \rightarrow 1} \|D_{ \varphi, u}^n h_a\|_{\mathcal{Z}_{\beta}} =0 $.
\item[$(iv)$] 
\begin{align*}
 & \limsup_{|\varphi (z)| \rightarrow 1} \frac{(1-|z|^2)^{\beta}}{(1-|\varphi(z)|^2)^{\alpha + n-2}} |u''(z)|=0, \\
 & \limsup_{|\varphi (z)| \rightarrow 1} \frac{ (1-|z|^2)^{\beta} }{(1-|\varphi(z)|^2)^{\alpha + n -1}} |2u'(z) \varphi'(z) + u(z) \varphi''(z) |=0, \\
 & \limsup_{|\varphi (z)| \rightarrow 1} \frac{ (1-|z|^2)^{\beta} }{(1-|\varphi(z)|^2)^{\alpha + n}} |u(z) \varphi'^2(z)|=0.
\end{align*}
\end{itemize}
\end{cor}
We next prove the result of Theorem \ref{g38} in the case of $n =1$ and $\alpha =1$.
\begin{theorem} \label{g40}
Let $0<\beta <\infty$ and  $D_{ \varphi, u}^1: \mathcal{Z} \rightarrow \mathcal{Z}_{\beta}$ be a bounded operator. Then,
\begin{align*}
\|D_{ \varphi, u}^1\|_{e,\mathcal{Z} \rightarrow  \mathcal{Z}_{\beta}} \approx \max \Big\{ & \limsup_{|\varphi(z)| \rightarrow 1} (1-|z|^2)^{\beta} |u''(z)| \log \frac{2} {(1-|\varphi(z)|^2)^{\alpha}},  \\
 & \limsup_{|\varphi(z)| \rightarrow 1} \frac{(1-|z|^2)^{\beta}}{1-|\varphi(z)|^2} |2u'(z) \varphi'(z) + u(z) \varphi''(z)|, \\
& \limsup_{|\varphi(z)| \rightarrow 1} \frac{(1-|z|^2)^{\beta}}{(1-|\varphi(z)|^2)^2} |u(z) \varphi'^2(z)| \Big\}.
\end{align*}
\end{theorem}
\begin{proof}
Let $\{ z_j \}$ be a sequence in $\mathbb{D}$ such that
$|\varphi(z_j)| \rightarrow 1$, as $j \rightarrow \infty$. Consider the test functions $ k_j = k_{\varphi(z_j)}$ and $l_j = l_{\varphi(z_j)} $ defined in the proof of Theorem \ref{g35}. Then, $\{k_j\}$ and $\{l_j\}$ are bounded sequences in $\mathcal{Z}_{0}$ which converge to zero on compact subsets of $\mathbb{D}$. Also,
$${k^{'}_j} (\varphi(z_j)) = {k^{''}_j} (\varphi(z_j)) =0 , \quad {k^{'''}_{j}}(\varphi(z_j)) = 16 \frac{|\varphi(z_j)|^{3}}{(1-|\varphi(z_j)|^2)^{2}},$$
and
$$ {l^{'}_j} (\varphi(z_j)) = {l^{'''}_j} (\varphi(z_j)) =0, \quad {l^{''}_{j}}(\varphi(z_j)) = -2 \frac{|\varphi(z_j)|^{2}}{1-|\varphi(z_j)|^2}.$$
Thus, for any compact operator $K: \mathcal{Z} \rightarrow \mathcal{Z}_{\beta}$, we have
\begin{align*}
\| D_{\varphi, u}^1 - K\|_{\mathcal{Z} \rightarrow \mathcal{Z}_{\beta}} \gtrsim & \limsup_{j \rightarrow \infty}
\| D_{\varphi, u}^1 (k_j)\|_{\mathcal{Z}_{\beta}} - \limsup_{j \rightarrow \infty} \| K (k_j)\|_{\mathcal{Z}_{\beta}} \\
\gtrsim  & \limsup_{j \rightarrow \infty}  \frac{(1-|z_j|^2)^{\beta} |\varphi(z_j)|^{3}}{(1-|\varphi(z_j)|^2)^{2}} |u(z_j)| |\varphi'^2(z_j)| \\
= & \limsup_{|\varphi(z)| \rightarrow 1}  \frac{(1-|z|^2)^{\beta}}{(1-|\varphi(z)|^2)^{2}} |u(z)| |\varphi'^2(z)|,
\end{align*}
and
\begin{align*}
\| D_{\varphi, u}^1 - K\|_{\mathcal{Z} \rightarrow \mathcal{Z}_{\beta}} \gtrsim & \limsup_{j \rightarrow \infty}
\| D_{\varphi, u}^1 (l_j)\|_{\mathcal{Z}_{\beta}} - \limsup_{j \rightarrow \infty} \| K (l_j)\|_{\mathcal{Z}_{\beta}} \\
\gtrsim  & \limsup_{j \rightarrow \infty}  \frac{(1-|z_j|^2)^{\beta} |\varphi(z_j)|^{2}}{1-|\varphi(z_j)|^2} |2 u'(z_j) \varphi'(z_j) + u(z_j)\varphi''(z_j)| \\
= & \limsup_{|\varphi(z)| \rightarrow 1} \frac{(1-|z|^2)^{\beta}}{1-|\varphi(z)|^2} |2 u'(z) \varphi'(z) + u(z)\varphi''(z)|.
\end{align*}
Therefore,
$$\| D_{\varphi, u}^1\|_{e, \mathcal{Z} \rightarrow \mathcal{Z}_{\beta}}  = \inf_{K} \| D_{\varphi, u}^1 - K\|_{\mathcal{Z} \rightarrow \mathcal{Z}_{\beta}}
\gtrsim \limsup_{|\varphi(z)| \rightarrow 1}  \frac{(1-|z|^2)^{\beta}}{(1-|\varphi(z)|^2)^{2}} |u(z)| |\varphi'^2(z)|,$$
and also
$$\| D_{\varphi, u}^1\|_{e, \mathcal{Z} \rightarrow \mathcal{Z}_{\beta}} \gtrsim \limsup_{|\varphi(z)| \rightarrow 1} \frac{(1-|z|^2)^{\beta}}{1-|\varphi(z)|^2} |2u'(z) \varphi'(z) + u(z) \varphi''(z)|.$$
Now, consider the test functions $t_j = t_{\varphi(z_j)}$ defined in the proof of Theorem \ref{g35}. Then, $\{t_j\}$ is a bounded sequence in $\mathcal{Z}_{0}$ which converges to zero on compact subsets of $\mathbb{D}$. Moreover,
$$t_j'(\varphi(z_j)) = \log \frac{1}{1-|\varphi(z_j)|^2}, \quad t_j''(\varphi(z_j)) = \frac{2 \overline{\varphi(z_j)}}{1-|\varphi(z_j)|^2},$$
$$t_j'''(\varphi(z_j)) = \frac{2 \overline{\varphi(z_j)}^2}{(1-|\varphi(z_j)|^2)^2} (1+ (\log \frac{1}{1-|\varphi(z_j)|^2})^{-1}).$$
Consequently,
\begin{align*}
\| D_{\varphi, u}^1 & \|_{e, \mathcal{Z} \rightarrow \mathcal{Z}_{\beta}}  =  \inf_{K} \| D_{\varphi, u}^1 - K\|_{\mathcal{Z} \rightarrow \mathcal{Z}_{\beta}} \\
\gtrsim & \limsup_{j \rightarrow \infty}  (1-|z_j|^2)^{\beta} |u''(z_j)| \log \frac{1}{1-|\varphi(z_j)|^2} \\
& - \limsup_{j \rightarrow \infty}  (1-|z_j|^2)^{\beta}  |2 u'(z_j) \varphi'(z_j) + u(z_j) \varphi''(z_j)| \frac{2 |\varphi(z_j)|}{1-|\varphi(z_j)|^2} \\
& - \limsup_{j \rightarrow \infty}  (1-|z_j|^2)^{\beta}|u(z_j) \varphi'^2 (z_j) | \frac{2 |\varphi(z_j)|^2}{(1-|\varphi(z_j)|^2)^2} (1+ (\log \frac{1}{1-|\varphi(z_j)|^2})^{-1}) \\
= & \limsup_{|\varphi(z)| \rightarrow 1}  (1-|z|^2)^{\beta} |u''(z)| \log \frac{1}{1-|\varphi(z)|^2} \\
& - \limsup_{|\varphi(z)| \rightarrow 1}   (1-|z|^2)^{\beta}  |2 u'(z) \varphi'(z) + u(z) \varphi''(z)| \frac{2}{1-|\varphi(z)|^2} \\
& - \limsup_{|\varphi(z)| \rightarrow 1}   (1-|z|^2)^{\beta}|u(z) \varphi'^2 (z) | \frac{2}{(1-|\varphi(z)|^2)^2} (1+ (\log \frac{1}{1-|\varphi(z)|^2})^{-1}) \\
\gtrsim &  \limsup_{|\varphi(z)| \rightarrow 1}  (1-|z|^2)^{\beta} |u''(z)| \log \frac{1}{1-|\varphi(z)|^2} - \| D_{\varphi, u}^1\|_{e, \mathcal{Z} \rightarrow \mathcal{Z}_{\beta}} -\| D_{\varphi, u}^1\|_{e, \mathcal{Z} \rightarrow \mathcal{Z}_{\beta}},
\end{align*}
which implies that
$$ \| D_{\varphi, u}^1 \|_{e, \mathcal{Z} \rightarrow \mathcal{Z}_{\beta}} \gtrsim \limsup_{|\varphi(z)| \rightarrow 1}  (1-|z|^2)^{\beta} |u''(z)| \log \frac{1}{1-|\varphi(z)|^2}.$$
The proof of upper estimate is similar to the proof of Theorem \ref{g38} using the operators $K_r$ and Lemma \ref{l31}.
\end{proof}
\begin{cor}
Let $0<\beta <\infty$ and  $D_{ \varphi, u}^1: \mathcal{Z} \rightarrow \mathcal{Z}_{\beta}$ be a bounded operator. Then, $D_{ \varphi, u}^1: \mathcal{Z} \rightarrow \mathcal{Z}_{\beta}$ is compact if and only if
\begin{itemize}
\item[$(i)$] $\limsup_{|\varphi(z)| \rightarrow 1} (1-|z|^2)^{\beta} |u''(z)| \log \frac{2} {(1-|\varphi(z)|^2)^{\alpha}}=0$,
\item[$(ii)$] $\limsup_{|\varphi(z)| \rightarrow 1} \frac{(1-|z|^2)^{\beta}}{1-|\varphi(z)|^2} |2u'(z) \varphi'(z) + u(z) \varphi''(z)|=0$,
\item[$(iii)$] $\limsup_{|\varphi(z)| \rightarrow 1} \frac{(1-|z|^2)^{\beta}}{(1-|\varphi(z)|^2)^2} |u(z) \varphi'^2(z)| =0$.
\end{itemize}
\end{cor}
\begin{remark}\label{rem1}
Montes-Rodr$\acute{\text{i}}$guez in \cite[Theorem 2.1]{mon}, and also Hyv$\ddot{\text{a}}$rinen et al. in \cite[Theorem 2.4]{lin1},
proved that if $\nu$ and $\omega$ are radial and non-increasing weights tending to zero at the boundary of $\mathbb{D}$, then
\begin{itemize}
\item[$(i)$] the weighted composition operator $u C_{\varphi}$ maps $H_{\nu}^{\infty}$ into $H_{\omega}^{\infty}$
if and only if
$$\sup_{n \geq 0} \frac{\|u \varphi^n\|_{\omega}}{\|z^n\|_{\nu}} \asymp \sup_{z \in \mathbb{D}} \frac{\omega(z)}{\widetilde{\nu}(\varphi(z))} |u(z)| < \infty, $$
with norm comparable to the above supremum.
\item[$(ii)$] $\|u C_{\varphi}\|_{e, H_{\nu}^{\infty} \rightarrow H_{\omega}^{\infty}} = \limsup_{n \rightarrow \infty} \frac{\|u \varphi^n\|_{\omega}}{\|z^n\|_{\nu}}
= \limsup_{|\varphi(z)| \rightarrow 1} \frac{\omega(z)}{\widetilde{\nu}(\varphi(z))} |u(z)|.$
\end{itemize}
Also, by \cite[Lemma 2.1]{lin3}, we know that for each $0 < \alpha < \infty$
\begin{itemize}
\item[$(iii)$] $\limsup_{n \rightarrow \infty} (n+1)^{\alpha} \|z^n\|_{{\nu}_{\alpha}} = (\frac{2 \alpha }{e})^{\alpha}$,
\item[$(iv)$] $\limsup_{n \rightarrow \infty} (\log n) \|z^n\|_{{\nu}_{\log}} =1$.
\end{itemize}
By applying these facts, our results in this paper containing terms of the type $\frac{\omega(z)}{\widetilde{\nu}(\varphi(z))} |u(z)|$ can be restated in terms of $u$ and $\varphi^n$. See, for example, \cite{SH} and references therein for these types of results.
\end{remark}

\end{document}